\documentclass[11pt]{epiarticle}
\usepackage{epimath}

\setpapertype{A4}

\usepackage[english]{babel}

\usepackage[dvips]{graphicx}     

\newtheorem{conjecture}[theorem]{Conjecture}

\newcommand{\nc}{\newcommand}

\def\hf{{\textstyle{1\over2}}}

\newcommand{\C}{{\mathbb C}}
\newcommand{\N}{{\mathbb N}}

\newcommand{\R}{{\mathbb R}}

\nc{\sinc}{\operatorname{sinc}}
\nc{\sgn}{\operatorname{sgn}}
\nc{\supp}{\operatorname{supp}}
\nc{\Real}{\operatorname{Re}}
\nc{\Imag}{\operatorname{Im}}
\nc{\dif}{\operatorname{d}}
 \nc{\im}{\operatorname{i}}
\nc{\Hi}{{\mathscr{H}}^\infty} \nc{\Ht}{{\mathscr{H}}^2}
\nc{\Hone}{{\mathscr{H}}^1} \nc{\ol}{\overline} \nc{\bz}{\mathbf{z}}
\nc{\bw}{\mathbf{w}} \nc{\eps}{\varepsilon}

\title{On certain sums over ordinates of  zeta-zeros II}
\titlemark{On certain sums over ordinates of  zeta-zeros II}
\author{Andriy Bondarenko \&\ Aleksandar Ivi\'{c} \&\ Eero Saksman \&\ Kristian Seip}
\authormark{A. Bondarenko \&\ A. Ivi\'{c} \&\ E. Saksman \&\ K. Seip}
\date{} 
\journal{Hardy-Ramanujan Journal -- (yyyy), ---} 
\acceptation{submitted dd/mm/yyyy, accepted dd/mm/yyyy, revised dd/mm/yyyy}

\begin{document}

\maketitle

\dedication{Dedicated to the memory of S. Srinivasan (1943 -- 2005)}


\thanks{Saksman's research was supported in part by the Lars Onsager Professorship at NTNU and in part by the Finnish Academy CoE ``Analysis and Dynamics''. The research of Bondarenko and Seip was supported in part by Grant 227768 of the Research Council of Norway.}

\begin{prelims}

\def\abstractname{Abstract}
\abstract{Let $\gamma$ denote the imaginary parts of  complex zeros $\rho = \beta+i\gamma$
of $\zeta(s)$. The problem of analytic continuation of the
 function $G(s) := \sum\limits_{\gamma > 0}\gamma^{-s}$ to the left of the line
 $\Re s = -1$ is investigated, and its Laurent expansion at the pole $s=1$
 is obtained.  Estimates for the second moment on the critical line $\int_1^T|G(\hf+it)|^2\dif t$ are revisited.
 This paper is a continuation of work begun by the second author in \cite{Iv}.}

 \\


\end{prelims}


\section{Introduction: the function $G(s)$}\label{se:G(s)}

We will be concerned with the Dirichlet series
\begin{equation}\label{eq:iv11}
G(s): = \sum_{\gamma > 0}\gamma^{-s},
\end{equation}
where $\gamma$ denotes the ordinates of the complex zeros of the Riemann zeta-function
$\zeta(s)$, counted as usual with multiplicities. Thus $G(s)$ converges
absolutely and represents a holomorphic function in the half-plane $\sigma = \Real s > 1$, in view of the classical
Riemann--von Mangoldt formula
\begin{equation}\label{eq:iv12}
N(T) = \frac{T}{2\pi}\log \frac{T}{2\pi} - \frac{T}{2\pi} + O(\log T)
\end{equation}
 for the zero counting function $N(T):= \sum_{0<\gamma\leqslant T}1$ (see \cite[p. 17]{I} or \cite[p. 214]{T}).
The condition $\gamma>0$ in \eqref{eq:iv11} is natural, since the property
$\overline{\zeta(s)} = \zeta(\bar{s})$ ensures that
$\beta-i\gamma$ is a zero of $\zeta(s)$ if $\beta+i\gamma$ is a zero.

The function $G(s)$ is mentioned in the work of Delsarte \cite{D} and,
in a perfunctory way, in the works of Chakravarty \cite{Ch1,Ch2}.  A related zeta function, namely
\[
\sum_{\gamma > 0}\gamma^{-s}\sin(\alpha\gamma)\qquad(\alpha > 0),
\]
was studied by Fujii \cite{Fu1}, but its properties are different from
those of $G(s)$, and we shall not consider it here. Both
Chakravarty and Delsarte (as well as Fujii) assume the Riemann hypothesis
(that all complex zeros of $\zeta(s)$ satisfy $\Real s = 1/2$)
when dealing with $G(s)$. Delsarte \cite[p. 430]{D} obtains its analytic continuation
to $\C$ by employing Poisson summation and a sort of
a modular relation.

We discuss analytic continuation of Delsarte's function $G(s)$ in  Section \ref{se:anco} We give a different proof of  Delsarte's result on the analytic continuation under the Riemann hypothesis. Moreover, we note that meromorphic extension to a larger half-plane $\sigma >-1-\varepsilon$ would have strong consequences related to the density hypothesis.
In Section \ref{se:laurent}, we provide an explicit Laurent expansion of $G(s)$ at $s=1$.
Section \ref{se:second} is devoted to estimating the second moment $\int_1^T|G(\hf+it)|^2\dif t$.
There  the growth of the second moment is connected explicitly to the
fluctuations of the function $S(t)$ (see \eqref{eq:iv14} below),
and we provide lower and upper bounds for the growth. Finally,
in Section \ref{se:mellin} we observe how the connection of $G(s)$
to the ``super zeta functions''  studied in the monograph
\cite{Vo}\footnote{The study of such zeta functions goes back
to Mellin \cite{M} in 1916 and seems not to be widely known.}  leads to yet another approach to Delsarte's result.

 \section{Analytic continuation of $G(s)$}\label{se:anco}

The second author introduced and studied $G(s)$ in \cite{Iv}.
Starting from the classical formula
\begin{equation}\label{eq:iv10}
N(T) = \frac{T}{2\pi}\log \frac{T}{2\pi} - \frac{T}{2\pi} +\frac{7}{8}+ S(T)+f(T),
\end{equation}
it was shown that, for $\sigma>0$ and a suitable constant $C_1$,
\begin{equation}\label{eq:iv13}
G(s) = {1\over2\pi(s-1)^2} - {\log(2\pi)\over2\pi(s-1)} + C_1
+ s\int_1^\infty\bigl(S(x) + f(x)\bigr)x^{-s-1}\dif x.
\end{equation}
Here, if $T$ is not an ordinate of a zero,
\begin{equation}\label{eq:iv14}
S(T) = {1\over\pi}\arg\zeta(\hf+iT) = {1\over\pi}\Imag\left\{
\log\zeta(\hf+iT)\right\} \ll \log T,
\end{equation}
where the argument of $\zeta(\hf+iT)$ is
obtained by continuous variation of the argument of $\zeta(s)$ along the straight lines joining the
points 2, $2+iT$, $\hf+iT$, starting with the value 0 at $s=2$. If $T$
is  an ordinate of a zero, then $S(T) = S(T+0)$. The function $f(T)$ is smooth
and admits an asymptotic expansion in terms of negative odd powers of $T$, namely
for any given integer $N \geqslant1 $ one has
\begin{equation}\label{eq:iv15}
f(T) = \sum_{j=1}^N \frac{a_j}{T^{2j-1}} + O_N\left(\frac{1}{T^{2N+1}}\right),
\end{equation}
where the $a_j$ are explicit constants, and $O_N$ means that the implied $O$-constant depends only on
$N$. From \eqref{eq:iv13}--\eqref{eq:iv15} one obtains the analytic continuation of $G(s)$ for $\sigma>0$.

Further analytic continuation
will follow by integrating the  integral in \eqref{eq:iv13} by parts.
This will give, for $\sigma > -1$ and a suitable constant $C_1$,
\begin{align}\label{eq:iv16}
G(s) &= \frac{1}{2\pi(s-1)^2} - \frac{\log2\pi}{2\pi(s-1)} + C_1
+ s\int_1^\infty f(x)x^{-s-1}\dif x \\
&\hbox{}\quad+ s(s+1)\int_1^\infty
\int_1^x S(u)\dif u\cdot x^{-s-2}\dif x,\nonumber
\end{align}
since we have the (unconditional) bound (see \cite[pp. 221--222]{T})
\begin{equation}\label{eq:iv17}
\int_0^T S(t)\dif t \;=\;O(\log T).
\end{equation}

To  study  the analytic continuation of $G(s)$ for $\sigma = \Real s \leqslant -1$ we need some notation.
Following \cite{Fu2}, we define
\[
{\tilde S}_0 (T) := S(T)\qquad(T\ne \gamma),
\]
and for $m\geqslant 1$,
\begin{equation}\label{eq:iv21}
{\tilde S}_m(T) := \int_0^T {\tilde S}
_{m-1} (t)\dif t + C_m,
\end{equation}
where $C_{2k} := (-1)^{k-1}/((2k)!4^k)$ and
\[
C_{2k-1} := \frac{(-1)^{k-1}}{\pi}\, \int_\frac12^\infty \int_{\sigma_1}^\infty\cdots
\int_{\sigma_{2k-2}}^\infty \log|\zeta(\sigma)| \dif \sigma_{2k-1} \cdots {\dif} \sigma_1
\]
for every positive integer $k$.
When $T =\gamma$, we put
$$
{\tilde S}_m(T) := \frac12\Bigl({\tilde S}_m(T+0) + {\tilde S}_m(T-0)\Bigr).
$$
On the Riemann hypothesis, it is known that
(see \cite[pp. 350--354]{T})
\begin{equation}\label{eq:iv22}
{\tilde S}_m(T) \;\ll_m \; \frac{\log T}{(\log\log T)^{m+1}}\qquad(m\geqslant1).
\end{equation}
Here $\ll_m $ means that the  constant implied by the $\ll$-symbol depends only on $m$. Note that
\begin{equation}\label{eq:iv23}
\int_1^\infty x^{-2j+1-s-1}\dif x \;=\; \frac{1}{2j+s-1}\qquad(\sigma > 1 -2j,\, j\in\N).
\end{equation}
Thus using \eqref{eq:iv15} and \eqref{eq:iv23}, we see
that
\[
s\int_1^\infty f(x)x^{-s-1}\dif x
\]
in \eqref{eq:iv16} represents a meromorphic function in $\C$, which has simple poles at $j = -1,-3, \ldots\;$.
Hence repeated integrations by parts of the last term in \eqref{eq:iv13} show, on using \eqref{eq:iv21}, that
under the Riemann hypothesis the function $G(s)$ admits meromorphic continuation to $\C$. Its only poles are:
$s = 1$ (order two) and $s = -1,-3, \ldots$ (order one).
\medskip

We next look at what are the consequences if we do not assume the Riemann hypothesis to hold,
but instead  \emph{assume} that $G(s)$ admits a meromorphic continuation across the line $\sigma=-1$.
The integration by parts argument used in the preceding discussion implies in particular that then the function
\begin{equation}\label{eq:H1}
H_1(s):=\int_1^\infty {\tilde S}_2(x)x^{-s-3}\dif x
\end{equation}
continues meromorphically across the same line. By \cite[Theorem 3]{Fu2}, we have
\begin{equation} \label{eq:St2}
{\tilde S}_2(T)\;=\; \sum_{\beta >1/2, 0<\gamma \leqslant T}(\beta -1/2)^2  \;\; +\; O(\log T) =A(T)\;\; +\; O(\log T),
\end{equation}
where we have set $A(T):=\sum_{\beta >1/2, 0<\gamma \leqslant T}(\beta -1/2)^2$. 
We therefore obtain that also the function
\[
H(s):=\int_1^\infty A(x)x^{-s-3}\dif x
\]
extends meromorphically across $\sigma=-1$.

Now let as usual $N(\sigma, T)$ be the number of zeros $\beta +i\gamma$ with
$\frac12 \leqslant \sigma<\beta<1$ and $0<\gamma\leqslant T$.
We may then write
\[ A(x)= 2\int_{1/2}^1 (\sigma-\textstyle\frac12) N(\sigma,x) \dif\sigma . \]
Namely the right-hand side above equals
\begin{align*}
2\int_{1/2}^1(\sigma-\textstyle\frac12)\sum\limits_{\gamma\leqslant x,\beta\geqslant\sigma\geqslant\frac12}1\dif\sigma
 & = 2\sum\limits_{\gamma\leqslant x}\int_{\frac12}^\beta(\sigma-\textstyle\frac12)\dif\sigma\\&
= \sum_{\gamma\leqslant x,\beta\geqslant\frac12}{(\beta-\textstyle\frac12)}^2 = A(x).
\end{align*}

By Selberg's uniform density estimate \cite[Theorem 9.19 (C)]{T},
\[
A(x)  \ll x \log x \int_0^{1/2} y x^{-y/4} \dif y  \ll \frac{x}{\log x}.
\]
Now
\[ \int_{e}^{\infty} \frac{x^{-\varepsilon-1}}{\log x}\dif x =\int_1^{\infty} \frac{e^{-\varepsilon u}}{u} \dif u
\leqslant \int_{1}^{1/\varepsilon}\frac{\dif u}{u}+\varepsilon
\int_{0}^{\infty}e^{-\varepsilon u}\dif u\leqslant \log \frac{1}{\varepsilon}+1.
\]
Hence $H(s)$ cannot have a pole at $s=-1$, and by our assumption it  thus has to be analytic at $s=-1$.  Then since $A$ is non-negative and increasing, a classical theorem of Landau \cite[Theorem 6, Chapter II.1]{Te}  implies that there is $\delta >0$ such that
\[ A(x)\ll x^{1-\delta}. \]
It follows from the definition of $A(x)$ and the fact that $\sigma\mapsto N(\sigma,x)$  is a decreasing function that
for $\sigma >1/2$,
\begin{equation}\label{eq:N}
 N(\sigma, T)\ll (\sigma-1/2)^{-2} T^{1-\delta},
 \end{equation}
with the implicit constant depending only on $\delta$.
For $\sigma$ close to $1/2$, this assertion is stronger than the density hypothesis
that $N(\sigma, T) \ll_\varepsilon T^{2-2\sigma+\varepsilon}\;(1/2 \leqslant \sigma \leqslant 1)$.

We summarize the preceding discussion as follows:

\begin{theorem}\label{th:iv1}
Unconditionally, $G(s)$ has meromorphic continuation at least
to $\sigma >-1$; it has only one pole in this half-plane, located at $s=1$ and of order two. If the Riemann hypothesis is true, then $G(s)$ extends meromorphically to $\C$ with additional poles that are all simple and located at $\,s = -1,-3, \ldots\,$.
If $G(s)$  continues meromorphically across $\sigma=-1$, then the strong implication \eqref{eq:N} holds for the density of the zeros off the critical line.
\end{theorem}

\begin{remark}\label{rem:AnCo}
As noted in Section \ref{se:G(s)}, the meromorphic continuation of $G(s)$ to $\C$ on the Riemann hypothesis was already
obtained by Delsarte \cite{D}, but the above proof is simpler (see also Section~\ref{se:mellin} below).
The last assertion of Theorem~\ref{th:iv1} indicates that establishing meromorphic continuation across the line $\sigma=-1$ is likely to be very difficult.
\end{remark}

Notice that we may have meromorphic extension to the whole complex plane even if the Riemann hypothesis fails. This will for instance be the case if there are finitely many nontrivial zeros $\rho=\beta+i\gamma$ with $\beta>1/2$. To see this, we write \eqref{eq:St2} in the form $\tilde{S}_2(T)=A(T)+I_2(T)$. Plugging this into \eqref{eq:H1}, we see that the term $A(x)$  gives rise to an additional simple pole at $-2$, while the term $I_2(x)$, by repeated integration by parts and use of \cite[Theorem 1]{Fu2}, yields an entire function.

An elaboration of this observation, along the lines of our discussion of possible meromorphic continuation across $\sigma=-1$, shows that the distinctive feature of the meromorphic continuation obtained on the Riemann hypothesis is that $G(s)$ will be analytic on the segment  $[-2,-1)$. This leads to the following statement.

\begin{theorem}
The Riemann hypothesis holds if and only if $G(s)$ has meromorphic continuation to the half-plane $\sigma\geqslant -2$ with only two poles that are located at respectively $1$ and $-1$.
\end{theorem}

\section{ The Laurent expansion of $G(s)$ at $s=1$}\label{se:laurent}

From \eqref{eq:iv13} one can obtain the (unconditional) Laurent expansion of $G(s)$ at its pole $s=1$. Namely, denoting
$F(x) := S(x) + f(x)$ for brevity, we have for $|s-1|<1$,
\begin{align}\label{eq:iv31}
&\int_1^\infty F(x)x^{-s-1}\dif x = \int_1^\infty F(x)e^{-(s-1)\log x}x^{-2}\dif x
\;=\; \sum_{j=0}^\infty c_j(s-1)^j,
\end{align}
say, with
\begin{equation}\label{eq:iv32}
c_j \;:=\; \frac{(-1)^j}{j!}\int_1^\infty F(x)\log^jx\cdot x^{-2}\dif x\qquad(j = 0, 1, \ldots)\,.
\end{equation}
Here we could interchange integration and summation in view of absolute convergence; this follows
from the bound $F(x) \ll \log x$ which holds because $S(x) \ll \log x$. Writing the last expression in \eqref{eq:iv13} as
\[
(s-1)\int_1^\infty F(x)x^{-s-1}\dif x + \int_1^\infty F(x)x^{-s-1}\dif x
\]
and using \eqref{eq:iv31}, we obtain
\begin{theorem}\label{th:iv2}
The Laurent expansion of $G(s)$ at $s=1$ has the form
\[
G(s) = {1\over2\pi(s-1)^2} - {\log2\pi\over2\pi(s-1)} + \sum_{j=0}^\infty b_j(s-1)^j,
\]
where
\[
b_0 = C_1 + c_0,\quad b_j = c_j + c_{j-1}\quad(j = 1,2, \ldots),
\]
and the $c_j$ are given by \eqref{eq:iv32}.
\end{theorem}

\section{Estimates for the second moment}\label{se:second}

We will now provide upper and lower bounds for the second moment $$
\int_1^T |G(\hf+it)|^2 \dif t,
$$
thus continuing the investigations started in  \cite[Section 2]{Iv}.
To this end, we follow \cite{Iv} and write
\[ G(s)= \sum_{\gamma\leqslant X} \gamma^{-s} + R(s), \]
where
\[ R(s)=\int_{X}^{\infty} \frac{x^{-s}}{2\pi} \log \frac{x}{2\pi}\dif x + \int_{X}^{\infty} x^{-s}\dif (S(x)+f(x)). \]
Integrating by parts, we get
\begin{equation} \label{eq:R} R(s) = \frac{X^{1-s}}{2\pi(s-1)}  \log \frac{X}{2\pi} + \frac{X^{1-s}}{2\pi(s-1)^2}
-X^{-s}\big(S(X)+f(X)\big)+ s \int_X^{\infty} x^{-s-1}(S(x)+f(x))\dif x, \end{equation}
which to begin with is valid for $\sigma >1$, but which continues meromorphically for all $\sigma>0$.
For $s=1/2+it$, the integral on the right-hand side of \eqref{eq:R} may be written as
$\widehat g(t/2\pi)$, where $g(\xi):= e^{-\xi/2}\big(S(e^{\xi})+f(e^{\xi})\big)\chi_{[X,\infty)}(e^{\xi})$.
Hence elementary estimates and Parseval's identity  together with Selberg's bound
$\int_T^{2T}S^2(t)\dif t\ll T\log\log T$ yield
\[ \int_{1}^T \big|R(\hf+it)\big|^2 \dif t \ll X \log^2X
+ (T^2/X) \log\log X. \]
Choosing $X=T \sqrt{\log\log T} /\log T$ we get
\begin{equation}\label{R}
 \int_{1}^T \big|R(\hf+it)\big|^2 \dif t \ll T \log T \sqrt{\log\log T}.
 \end{equation}
As we  will see in \eqref{diag} below, $\int_{1}^T |\sum_{\gamma\leqslant X}
\gamma^{-1/2-it} |^2 \dif t \gg T \log^2 T$ when $X\geqslant T^b$ for some positive number $b$.
Hence the second moment of the partial sum $\sum_{\gamma\le X} \gamma^{-1/2-it}$ yields
the dominant term which needs to be studied in more detail.

Let $\psi$ be a symmetric function in  $C(\R)$ satisfying\footnote{We may for example choose
$\psi(x)=(\max (1-|x|),0)^2$.}
\[ 0\leqslant\psi(x)\leqslant\chi_{[-1,1]}(x)\quad \text{and}
\quad 0<c_1\leqslant (1+4\pi^2|\xi|)^2\widehat\psi(\xi)\leqslant c_2\] for all $\xi$ in $\R$.
Let $A$ be an arbitrary subset of $\{\gamma:\; 0<\gamma\leqslant T\}$. We find that
\begin{align*}
\int_{-T}^T\big|\sum_{\gamma \in A}\gamma^{-1/2-it}\big|^2\dif t & \geqslant
\int_{\R}\big|\sum_{\gamma \in A}\gamma^{-1/2-it}\big|^2\psi(t/T)\dif t
=T\sum_{\gamma,\gamma'\in A}(\gamma\gamma')^{-1/2}\widehat \psi( T\log(\gamma/\gamma')/2\pi)\\
& \asymp T\sum_{\gamma,\gamma'\in A}(\gamma\gamma')^{-1/2}\frac{1}{1+T^2(\log(\gamma'/\gamma))^2}.
\end{align*}
Fix an arbitrary constant $a>0$ and  replace $\psi(x)$ by  $\widehat \psi(x/(2\pi a))$ in the above computation.
We obviously obtain an upper estimate up to a constant, but as
  $\psi(a\cdot )\ll \widehat\psi(\cdot/(2\pi))\asymp  \widehat\psi(a\cdot/(2\pi))$, we infer that
\begin{align}\label{fourier}
\int_{-T}^T\big|\sum_{\gamma \in A}\gamma^{-1/2-it}\big|^2\dif t
&\asymp T\sum_{\gamma,\gamma'\in A}(\gamma\gamma')^{-1/2}\frac{1}{1+T^2(\log(\gamma'/\gamma))^2}\\
&\asymp T\sum_{\gamma,\gamma'\in A}(\gamma\gamma')^{-1/2}\psi\big(Ta\log(\gamma'/\gamma)\big)\label{fourier2} \\
&\asymp T\sum_{\gamma \in A}\gamma^{-1}\#\left\{\gamma'\in (0,T)\cap A\; :\; |\gamma'-\gamma|\leqslant \gamma/(aT)\right\}.\nonumber
\end{align}
As the parameter $a$ is arbitrary, we see that there is some flexibility in choosing the size of the ``window'' in \eqref{fourier2} yielding the pairs of ordinates that  contribute significantly to the second moment.

By considering  only   the diagonal  $\gamma=\gamma'$ in \eqref{fourier} and using the known average density of the ordinates, we obtain
\begin{equation}\label{diag}
\int_{-T}^T\big|\sum_{0<\gamma \leqslant X}\gamma^{-1/2-it}\big|^2\dif t\; \gg
\;T \sum_{0<\gamma\le T^{b}}\gamma^{-1}\gg T\log^2 T
\end{equation}
whenever $X\geqslant T^b$ for some exponent $b$ in $(0,1)$.

Our next goal is to express \eqref{fourier} in terms of either of the functions $N(t)$
and $S(t)$ in the special case when $A=\{ \gamma: 0<\gamma\leqslant X\}$.
We start with the basic formula involving $N(t)$.
\begin{lemma}\label{le:termsofS} Suppose that $100 \leqslant X\leqslant  T$. Then
\begin{equation}\label{N(t)}
\int_1^T\bigg|\sum_{0<\gamma\leqslant X}\gamma ^{-1/2-it}\bigg|^2\dif t
\; \asymp\;  \int_1^X\Bigg(\frac{N(t+t/T)-N(t)}{t/T}\Bigg)^2 \dif t .
\end{equation}
\end{lemma}
\begin{proof}
We start from \eqref{fourier} in the form
\begin{equation} \label{fourier2} \int_{-T}^T\big|\sum_{0< \gamma \leqslant X}\gamma^{-1/2-it}\big|^2\dif t
\asymp T\sum_{1<\gamma\le X}\gamma^{-1}\#\left\{\gamma'\; :\; |\gamma'-\gamma|\leqslant \gamma/(aT)\right\}.
\end{equation}
Here we discarded the condition that $\gamma'\leqslant X$ in the sum on the right-hand side.
The right-hand side has thus been increased by a term of size $O((\log X)^2 T/X)$ which is
admissible because the left-hand side is $\gg T \log^2 X$.
We have also taken into account that $\gamma>0$ implies $\gamma>1$.

Set $k_0:= \lfloor T\log X\rfloor$ and $b_k:=X^{k/k_0}$. We divide $(1,X]$ into intervals $I_k=(b_{k-1},b_k]$ for $k=1,\ldots, k_0+1$;  it will be convenient to agree that $I_0=\emptyset.$  For any real interval $I$ let $n(I)$ denote the number of zeta zeros with ordinates in $I$. Then for an arbitrary $\gamma$ in $I_k$ we obtain
\[
\#\big\{\gamma'\in (0,T)\; :\; |\gamma'-\gamma|< \gamma/T\big\}\; \leqslant\;  n(I_{k-1})+n(I_k)+n(I_{k+1}).
\]
It follows that
\begin{align*}
\sum_{\gamma\in (b_{k-1},b_k]}\gamma^{-1}\#\big\{\gamma'\in (0,T)\; :\; |\gamma'-\gamma|\leqslant \gamma/T\big\}
\quad \leqslant \quad  &b_{k-1}^{-1}n(I_k)\big(n(I_{k-1})+n(I_k)+n(I_{k+1})\big)
\\ \ll \;\;\;&b_{k-1}^{-1}n(I_{k-1})^2+b_{k}^{-1}n(I_{k})^2+b_{k+1}^{-1}n(I_{k+1})^2.
\end{align*}
On the other hand, we clearly have
$$
\sum_{\gamma\in (b_{k-1},b_k]}\gamma^{-1}\#\big\{\gamma'\in (0,T): |\gamma'-\gamma|\leqslant 2\gamma/T\big\}
\;\geqslant \;  b_{k}^{-1}n(I_{k})^2.
$$
Summing  both the upper and lower estimates over $k$ and using \eqref{fourier2}, we get
\begin{equation}\label{help6}
\int_{-T}^T\big|\sum_{0< \gamma \leqslant X}\gamma^{-1/2-it}\big|^2\dif t \; \asymp\;
 T\sum_{k}b_k^{-1}n(I_{k})^2.
\end{equation}
We may neglect the first interval and consider $I_k$ with $k\geqslant 2$.
Divide each $I_k$ into three equally long subintervals. One of these three
intervals contains at least $n(I_{k})/3$ ordinates, and there is then a subinterval
 $I'_k$ of $I_k\cup I_{k-1}$ with  $|I'_k|\geq |I_k|/3$ such that
 $ N(t+t/T)- N(t)\geq n(I_{k})/3$ for every $t$ in $I'_k$.
 Since $|I_k| \sim b_k/T\asymp t/T$ for $t$ in $2I_k$ we thus obtain
\begin{align*}
Tb_{k}^{-1}n(I_{k})^2\; & \ll \; Tb_k^{-1}|I_k| ^{-1}\int_{I_{k-1}\cup I_k}( N(t+t/T)- N(t))^2)\dif t \\
& \ll \;\int_{I_{k-1}\cup I_k}( N(t+t/T)- N(t))^2)\frac{\dif tT^2}{t^2}
\ll \; Tb_{k}^{-1}(n(I_{k-1})+n(I_{k})+n(I_{k+1}))^2,
\end{align*}
where the last inequality follows on noting that $ N(t+t/T)- N(t)\leqslant n(I_{k})+n(I_{k+1})$ for
interior points $t$ in $I_k$. Summing over $k$ and using \eqref{help6}, we finally arrive at \eqref{N(t)}.
\end{proof}
We may alternatively replace $N(t)$ by $S(t)$ in \eqref{N(t)}:
\begin{align}\label{S(t)}
\int_1^T\bigg|\sum_{0<\gamma\leqslant X}\gamma ^{-1/2-it}\bigg|^2\dif t
\; \asymp\;  \int_1^X\Bigg(\frac{S(t+t/T)-S(t)}{t/T}\Bigg)^2 \dif t .
\end{align}
Indeed, by \eqref{eq:iv10}, \eqref{N(t)} implies that
\[ \int_1^T\bigg|\sum_{0<\gamma\leqslant X}\gamma ^{-1/2-it}\bigg|^2\dif t
\; \ll \; \int_1^X\Bigg(\frac{S(t+t/T)-S(t)}{t/T}\Bigg)^2 \dif t + X \log^2 X. \]
Since the left-hand side is $\gg$ $T\log^2 X$, we get that the second
term on the right-hand side can be removed. Since \eqref{eq:iv10} implies that
\[
N(t+t/T) - N(t) \geqslant S(t+t/T)-S(t),
\]
it is also clear that \eqref{N(t)} trivially implies
\[
\int_1^T\bigg|\sum_{0<\gamma\leqslant X}\gamma ^{-1/2-it}\bigg|^2\dif t
\; \gg \;  \int_1^X\Bigg(\frac{S(t+t/T)-S(t)}{t/T}\Bigg)^2 \dif t .
\]

The above lemma as expressed by \eqref{S(t)} shows that the asymptotics
that we are interested in, depends crucially on the oscillations of $S(t)$ in rather small intervals.
To gain more insight  into this asymptotics, we will next take a closer look at  the size of the dyadic parts
\[
D(x):=\int_1^T \Big|\sum_{x< \gamma \leqslant  2x} \gamma^{-1/2-it} \Big|^2 \dif t .
\]
\begin{proposition}\label{pr:dyadic} {\rm (i)}\quad Assume that $T/\log T\leq x\leqslant T$. Then
$$
x\log^2 T\; \ll\;  D(x)\; \ll \; T\log^2T.
$$
\noindent  {\rm (ii)}\quad Assume that $100\leqslant x< T/\log T.$ Then
$$
T\log x\; \ll \; D(x)\; \ll \; T\log x \sqrt{\log\log x}.
$$
\end{proposition}
\begin{proof}[Proof of part (i) of Proposition~\ref{pr:dyadic}]
We may assume that $T\gg 1$ is an integer.  Divide the interval $(x,2x]$ into
$T$ subintervals $I_j:=(x+(j-1)x/T, x+jx/T]$ for $1\leqslant j \le T$. By \eqref{fourier},
arguing as in the proof of  \eqref{help6}, we then get
\begin{equation} \label{eq:sum}
D(x)\;\asymp\;\; \frac{T}{x}\sum_{x\leqslant \gamma<2x} \left| \left\{\gamma': |\gamma-\gamma'|
\leqslant \frac{x}{T} \right\}\right| \asymp \frac{T}{x}\sum_{j=1}^T n(I_j)^2.\
\end{equation}
Since we know that $\sum_{j=1}^T n(I_j)=n([x,2x))\asymp x\log x $,
an application of the Cauchy--Schwarz inequality to \eqref{eq:sum}  yields immmediately the lower bound
$D(x)\leqslant x\log^2x\asymp x\log^2 T.$  To obtain the upper bound,
we divide the ordinates in $[x,2x)$ into  $O\left(\log T\right)$ 1-separated sets,
and apply the Montgomery--Vaughan inequality \cite[Corollary 2]{MV} to each of the corresponding sums.
\end{proof}
In the proof of  the upper bound of part (ii), we will use the following lemma.
\begin{lemma}\label{lem:S}
There is a universal constant $c>0$ such that  if $n(I)\geqslant |I| \log x$ for a  subinterval  $I=[a,b]$ of $[x, 2x+x/T]$, then
\[ \left|\left\{t\in 2I: |S(t)|\geqslant n(I)/8 \right\}\right|\geqslant c|I|. \]
\end{lemma}
\begin{proof}
By the Riemann--von Mangoldt formula, the assumption $n(I)\geqslant |I| \log x$ implies that
\[ S(b)-S(a)\geqslant \frac{1}{2} |I| \log x . \]
This implies that if $S(b)\leqslant \frac{1}{4} |I| \log x$, then $S(a)\le - \frac{1}{4} |I| \log x$
in which case may choose a positive constant $c$ independent of $I$ such that the desired bound
$|S(t)|\geqslant n(I)/8$ holds on the interval $[a-c|I| \log x, a]$. On the other hand, if $S(b)> \frac{1}{4} |I| \log x$,
then $|S(t)|\geqslant n(I)/8$
holds on the interval  $[b, b+c|I| \log x]$.
\end{proof}

\begin{proof}[Proof of part (ii) of Proposition~\ref{pr:dyadic}] We divide the interval $(x,2x]$
into the same subintervals $I_j$ as in the preceding case. The lower bound of part is again proved
by the Cauchy--Schwarz inequality. Indeed, observing that the number of non-zero summands in the sum
over $j$ is at most of order $x\log x$, we find that
\begin{eqnarray*}
D(x)\;\asymp\;\; \frac{T}{x^2\log x} x\log x\sum_{j=1}^T n(I_j)^2\;\;\geqslant
\;\;\frac{T}{x^2\log x} (x\log x)^2\;\; =\;\; T\log x.
\end{eqnarray*}

We turn to the proof of the upper bound. To begin with, note that since
$x\leqslant T/\log T$, we have $|I_j| \log x\le 1$. This means that $I_j$
satisfies the condition of Lemma~\ref{lem:S} whenever  $n(I_j)\geqslant 1$. This fact allows us to make the following construction. Let $j_1$ be the smallest
$j$ such that $n(I_j)\geqslant 1$, and let $k_1$ be the largest $k$ such that
\[ n\left(\bigcup_{m=0}^{k-1} I_{j_1+m} \right)\geqslant  \frac{k x}{T} \log x. \]
Now choose $j_{\ell}$ and $k_{\ell}$ inductively such that $j_{\ell}$ is the smallest $j$
such that $j\geqslant j_{\ell-1}+k_{\ell-1}$ and $n(I_j)\geqslant 1$ while $k_\ell$ is the largest $k$ such that
\[ n\left(\bigcup_{m=0}^{k-1} I_{j_\ell+m} \right)\geqslant  \frac{k x}{T} \log x. \]
We terminate this iteration when $j_{\ell}+k_{\ell}\geqslant 2T$. We set
\[ J_{\ell}:= \bigcup_{m=0}^{k_{\ell}-1} I_{j_\ell+m} \]
and notice that $|J_{\ell}| \ll 1$ since
\[ n\left(I \right)=\frac{|I|}{2\pi}\log x + O(\log x). \]
It is also clear, by the maximality of $J_{\ell}$, that
\begin{equation} \label{eq:max} n(J_{\ell}) = (1+o(1))|J_{\ell}| \log x. \end{equation}

Plainly by the construction, since $x\leqslant T/\log T$, all the ordinates $\gamma $ in $[x,2x]$
will be covered by the intervals $J_{\ell}$, and we thus have
\begin{equation} \label{eq:ex} \sum_j n(I_j)^2 \leqslant \sum_{\ell} n(J_{\ell})^2. \end{equation}
Now, by a standard covering argument (see e.g. \cite[Proof of Lemma 4.4]{Ga}),
we can find a subcollection $\{ 2J'\}$ of the intervals $2J_{\ell}$\footnote{If $I\subset \R$
is an interval, $2I$ stands for the interval with same center and twice the length.}
such that no point in $[x,2x]$ belongs to more than two intervals $2J'$ and such that the union of  $\{ 2J'\}$ covers all the intervals $J_\ell$. In particular, we have
\[ \sum_{\ell} |J_{\ell} |^2 \ll \sum_{J'} |J'|^2. \]
Combining this with Lemma~\ref{lem:S} and \eqref{eq:max}, we get
\begin{align*} \sum_{\ell} n(J_{\ell})^2 \;&\ll  \;
 \log x  \sum_{\ell} (\log x) |J_{\ell}| \cdot  |J_{\ell}|
 \ll \log x  \sum_{J'} (\log x) |J'| \cdot  |J'| \\
& \ll \;\log x \int_x^{2x} |S(t)| dt \ll x \log x \sqrt{\log\log x} ,
 \end{align*}
where in the last step we used the Cauchy--Schwarz inequality and Selberg's mean value theorem\footnote{Alternatively, we could appeal to Ghosh's asymptotic formula for the first moment of $S(t)$ \cite{G}.} for $S(t)$.  Recalling \eqref{eq:ex} and \eqref{eq:sum}, we obtain the upper bound of part (ii).
\end{proof}

As an immediate consequence of the preceding analysis, we now obtain a mean square estimate
\footnote{The upper bound  $|G(\hf+it)|^2 \dif t \;\;  \ll\;\; T\log^2T$ was stated in \cite{Iv}
but with an  incomplete proof.}  for the function $G(s)$ on the critical line:
\begin{theorem}\label{th:bounds}
\quad $\displaystyle
T\log^2 T\;\;  \ll\;\; \int_1^T \left|G(\hf+it)\right|^2 \dif t \;\;  \ll\;\; T\log^2 T\sqrt{\log\log T}.
$
\end{theorem}
\begin{proof}
We choose the length $X$ of the partial sum $\sum_{0<\gamma\leqslant X} \gamma^{-s}$ to be
$T\sqrt{\log \log T}/\log T$ so that the contribution from the remainder term is given by \eqref{R}.
By \eqref{help6}  (or directly by \eqref{fourier}) there is quasi-orthogonality between the
 dyadic parts of the sum (i.e., dyadic parts that are not adjacent are orthogonal with respect to a weighted norm,
 where the size of the weight is approximately $(1+|t/T|^2)$).
 We apply part (i) of Proposition~\ref{pr:dyadic} to suitable dyadic
 subintervals of $[T/\log T, T\sqrt{\log\log T}/\log T]$ and part (ii) to subintervals of $[1,T/\log T]$.
 Summing over all these dyadic intervals, we obtain the desired upper bound. The lower bound
 is a consequence of \eqref{diag} and \eqref{R}.
\end{proof}

We expect, partly in view of the very regular behavior of the ordinates predicted
by the random matrix analogy, that the true order for dyadic summands is akin to something like
\begin{equation}\label{dyad_est}
D(x) \asymp\begin{cases} x\log^2T & \textrm{for}\; T/\log T \leqslant x\leqslant T,\\
 T\log  x & \textrm{for}\; x\leqslant T/\log T.
\end{cases}
\end{equation}
We note that this prediction corresponds to the lower bounds in Proposition \ref{pr:dyadic}.
This leads us to the following
\begin{conjecture}\label{con:L^2}
\quad $\displaystyle\int_1^T |G(\hf+it)|^2 \dif t \;\;  \asymp\;\; T\log^2 T.$
\end{conjecture}

\begin{remark}\label{rem:moments}
On the Riemann hypothesis, we may split the sequence of ordinates $\gamma$
in the range $T/\log T\leqslant \gamma \le T$ in $O\left(\log T/\log\log T\right)$
sequences that are all separated with separation constant $\geqslant 1/\log\log T$.
Using again the Montgomery--Vaughan inequality, we then obtain the conjectured bound for the sum over $\gamma$ in this restricted range, namely
\[ \int_{1}^{T} \Bigl|\sum_{T/\log T\leqslant \gamma \leqslant T} \gamma^{-1/2 -it} \Bigr|^2 \dif t
\;\ll\;  T\log^2 T . \]
\end{remark}

\section{On super zeta functions }\label{se:mellin}

In 1917 Mellin  \cite{M} considered  a class of zeta functions that are defined by using the Riemann zeta zeros as the building block of new Dirichlet series. For example, he established meromorphic continuation to the whole complex plane for a class of such functions including the function
$$
Z(s):=\sum_\rho{\rho^{-s}},
$$
where the sum is taken over all nontrivial zeros of the Riemann zeta function. Voros  has devoted a monograph \cite{Vo} on  generalisations of such functions, for which he  coins the name
\emph{super zeta functions}\footnote{We urge the reader to take a look at  Frankenhuijsen's MathSciNet review of the book, where Serge Lang's point of view of ``climbing the ladder'' is explained.}.

For the convenience of the reader, we now  sketch a proof of the existence of a meromorphic extension of the function
\begin{equation}\label{eq:mel}
 M_\alpha(s)\;:=\;\sum_{\rho} (\alpha-\rho)^{-s}.
 \end{equation}
to the whole complex plane. This function is initially well-defined in the half-plane $\sigma >1$, as one sets $(\alpha-w)^{-s}:=e^{-s \log (\alpha-w)}$
for $w$ in $\mathbb{C}\setminus [\alpha, +\infty)$, using the convention that $\log(\alpha-w)$ is real-valued for
real $w<\alpha$.
 We will will consider only  $M_\alpha$
with  $\alpha >-2$, although we could consider the more general case.
However, our main interest lies in the case $\alpha= 1/2$, which under the Riemann hypothesis yields another approach to Delsarte's result. We refer to \cite[Section 7]{Vo} which contains a detailed discussion of the analytic continuation,
and also note that \cite[Appendix D]{Vo} provides an English translation
of Mellin's work \cite{M} which was originally written in German.

Clearly \eqref{eq:mel} defines a function analytic in the half-plane $\sigma>1$.  Set $c:=\min(1, \alpha)$ and
let $\eta_{\alpha}$ be the ``Hankel contour'' going from $+\infty$ below the
real line to $c-i\varepsilon$, where it follows the semicircle of radius $\varepsilon<2+\alpha$ centred at $c$ to the left of $c$ to $c+i\varepsilon$, and then going back to $+\infty$. Then the integral
$$
E_\alpha(s)\; :=\;  \frac{1}{2\pi i} \int_{\eta_{\alpha}} \frac{\zeta'(w)}{\zeta(w)} (\alpha-w)^{-s} \dif w
$$
defines an entire function $E_\alpha$, by the rapid decay of $\zeta'(w)/\zeta(w)$.
For $s>1$, the contour can be moved into the left half-plane and towards $-\infty$ in the usual way and therefore
\begin{equation} \label{eq:cont} \zeta_{\alpha}(s)=E_{\alpha}(s)-\sum_{n=0}^{\infty} (\alpha+2+2n)^{-s}=E_{\alpha}(s)-2^{-s}\zeta(s,\alpha/2+1), \end{equation}
where as usual $\zeta(s,q)$ is the Hurwitz zeta function.
The meromorphic continuation follows since the Hurwitz zeta function is known to
be a meromorphic function in $\mathbb{C}$ which is analytic in $\mathbb{C}\setminus\{1\}$
with a simple pole of residue $1$ at $s=1$.

\begin{remark}\label{rem:mellon} We notice that in the case $\alpha>1$,
we may write
\[ E_{\alpha}(s)\; =\;(\alpha-1)^{-s}+ \frac{1}{2\pi i} \int_{\eta'_{\alpha}} \frac{\zeta'(w)}{\zeta(w)} (\alpha-w)^{-s} \dif w,\]
where the modified contour $\eta'_{\alpha}$ stays to the right of the pole at $1$.
We may then substitute the Dirichlet series for $\zeta'(s)/\zeta(s)$ in the latter integral and obtain
the representation (see \cite[formula (7.28)]{Vo})
\[
\zeta_{\alpha}(s)\;\; =\;(\alpha-1)^{-s}\; -\frac{1}{\Gamma(s)}\sum_{n=2}^{\infty} \frac{\Lambda(n) (\log n)^{s-1}}{n^{\alpha}} \;\; - \;\;2^{-s}\zeta(s,\alpha/2+1),
\]
as can be seen by employing the formula
\begin{equation}\label{eq:integral}
\frac{1}{2\pi i} \int_{\eta_{\alpha}} n^{-w} (\alpha-w)^{-s}\dif w\; =\; \frac{-(\log n)^{s-1}}{n^{\alpha}\Gamma (s)}.
\end{equation}
To verify \eqref{eq:integral}, we may first assume that $0<s<1$ since the general
case then follows by analytic continuation. When $0<s<1$, we may push the integral to the interval $[\alpha,\infty)$, which produces an additional factor of $-2\sin(\pi s).$ The integral then reduces  to the definition of $\Gamma(1-s)$, and the rest follows by Euler's reflection formula
$$
\sin(\pi s) \;=\; \frac{\pi}{ \Gamma(s)\Gamma(1-s)}.
$$

We note also that other functions than $\rho^{-s}$ may be chosen. As an example,
we  may take $F(s):=\exp\bigl((\alpha-s)^c\bigr)$ with $0<c<1$.
Then a modification of the above  augment verifies that the  ``zeta function''
\[ Z_F(s):=\sum_{\rho} {F(\rho)}^{-s} \]
has an analytic extension to $\mathbb{C}$.
\end{remark}

\bibliographystyle{amsalpha}

\authoraddresses{
Andriy Bondarenko \\
Department of Mathematical Sciences, Norwegian University of Science and Technology (NTNU), NO-7491 Trondheim, Norway \\
\email andriybond@gmail.com

Aleksandar Ivi\'c \\
Serbian Academy of Sciences and Arts, Knez Mihailova 35, 11000 Beograd,
Serbia \\
\email aleksandar.ivic@rgf.bg.ac.rs

Eero Saksman \\
Department of Mathematical Sciences, Norwegian University of Science and Technology (NTNU),
NO-7491 Trondheim, Norway {\rm and} \\ University of Helsinki,
Department of Mathematics and Statistics, P.O. Box 68, FIN-00014 University of Helsinki, Finland \\
\email eero.saksman@helsinki.fi

Kristian Seip \\
Department of Mathematical Sciences, Norwegian University of Science and Technology (NTNU),
NO-7491 Trondheim, Norway \\
\email kristian.seip@ntnu.no}

\end{document}